\documentclass[a4paper,twoside,bibtotoc,idxtotoc]{scrartcl}
\usepackage{amsmath,amssymb,amsfonts,amsthm} 
\usepackage{graphics,graphicx}                 
\usepackage{color}                    
\usepackage{hyperref,fancyhdr}                 
\usepackage[all]{xy}
\usepackage[applemac]{inputenc}
\usepackage[ngerman,english]{babel}
\usepackage{url}
\usepackage{abstract}
\usepackage[split]{splitidx}\makeindex

\newindex[Index of Notations]{n}
\newindex[Index of Terms]{idx} 

\renewcommand{\sectionmark}[1]{\markright{Free Group Actions from the Viewpoint of Dynamical Systems}}
\fancypagestyle{plain}{\fancyhf{}
\lhead[\fancyplain{}{\thepage}]{}
\rhead[]{\fancyplain{}{\thepage}}
}

\newtheorem{theorem}{Theorem}[section]
\newtheorem{proposition}[theorem]{Proposition}
\newtheorem{corollary}[theorem]{Corollary}
\newtheorem{lemma}[theorem]{Lemma}

\newtheorem{theorem-example}[theorem]{Theorem $\backslash$ Example}
\renewenvironment{proof}{\begin{sloppypar}\noindent{\bf Proof}}{\hfill$\blacksquare$\end{sloppypar}}
\theoremstyle{definition}
\newtheorem{definition}[theorem]{Definition}
\newtheorem{example}[theorem]{Example}

\newtheorem{remark}[theorem]{Remark}

\newtheorem{open problem}[theorem]{Open Problem}
\DeclareMathOperator{\Aut}{Aut}

\DeclareMathOperator{\id}{id}
\DeclareMathOperator{\ev}{ev}
\DeclareMathOperator{\GL}{GL}

\DeclareMathOperator{\Hom}{Hom}

\DeclareMathOperator{\pr}{pr}

\DeclareMathOperator{\supp}{supp}

\DeclareMathOperator{\Prim}{Prim}
\DeclareMathOperator{\Lin}{Lin}

\begin{document}

\title{\bf{Free Group Actions from the\\Viewpoint of Dynamical Systems}}
\author{Stefan Wagner\\
SFB 878 WWU M\"unster\\
\url{swagn_02@uni-muenster.de}}

\maketitle

\begin{abstract}
\noindent
A dynamical system is a triple $(A,G,\alpha)$, consisting of a unital locally convex algebra $A$, a topological group $G$ and a group homomorphism $\alpha:G\rightarrow\Aut(A)$, which induces a continuous action of $G$ on $A$. In this paper we present a new characterization of free group actions (in classical differential geometry), involving dynamical systems and representations of the corresponding transformation groups. Indeed, given a dynamical system $(A,G,\alpha)$, we provide conditions including the existence of ``sufficiently many" representations of $G$ which ensure that the corresponding action
\[\sigma:\Gamma_A\times G\rightarrow\Gamma_A,\,\,\,(\chi,g)\mapsto\chi\circ\alpha(g)
\]of $G$ on the spectrum $\Gamma_A$ of $A$ is free. In particular, the case of compact abelian groups is discussed very carefully. We further present an application to the structure theory of $C^*$-algebras and an application to the noncommutative geometry of principal bundles.
\end{abstract}

\pagenumbering{arabic}

\thispagestyle{empty}

\tableofcontents

\section{Introduction}

Since the \emph{Erlanger Programm} of Felix Klein, the defining concept in the study of a geometry has been its symmetry group. In classical differential geometry the symmetries of a manifold are measured by Lie groups, i.e., one studies smooth group actions of a Lie group $G$ acting by diffeomorphisms on a manifold $M$. Of particular interest is the class of smooth group actions which are free and proper: In fact, by a classical result having a free and proper action of a Lie group $G$ on a manifold $P$ is equivalent to saying that $P$ carries the structure of a principal bundle with structure group $G$. 

\noindent
The origin of this paper is the question of whether there is a way to translate the geometric concept of principal bundles to noncommutative differential geometry. From a geometrical point of view it is, so far, not sufficiently well understood what a ``noncommutative principal bundle" should be. Still, there are several approaches towards the noncommutative geometry of principal bundles: For example, there is a well-developed abstract algebraic approach known as Hopf-Galois extensions which uses the theory of Hopf algebras (cf. [Sch04]). Another topologically oriented approach can be found in [ENOO09]; here the authors use $C^*$-algebraic methods to develop a theory of principal noncommutative torus bundles. In [Wa11a] we have developed a geometrically oriented approach to the noncommutative geometry of principal bundles based on dynamical systems and the representation theory of the corresponding transformation groups.

\noindent
The starting point of the last approach is the observation that (smooth) group actions may also be studied from the viewpoint of dynamical systems (we will see soon that each group action induces a dynamical system and vice versa). Since we are interested in principal bundles, i.e., in free and proper smooth group actions, it is reasonable to ask if there exist natural algebraic conditions on a dynamical system $(A,G,\alpha)$ which ensure that the corresponding action
\[\sigma:\Gamma_A\times G\rightarrow\Gamma_A,\,\,\,(\chi,g)\mapsto\chi\circ\alpha(g)
\]of $G$ on the spectrum $\Gamma_A$ of $A$ is free. An important remark in this context is that the freeness condition of a group action (let's say of a group $G$) is pretty similar to the condition appearing in the definition of a family of point separating representations of $G$. 

\noindent
In fact, our conclusions involve the following main observation: Given a dynamical system $(A,G,\alpha)$ and a family $(\pi_j,V_j)_{j\in J}$ of point separating representations of $G$, we can associate the ``generalized spaces of sections" $\Gamma_A V_j:=(A\otimes V_j)^G$ and it turns out that if $A$ is commutative and the evaluation map
\[\ev^j_{\chi}:\Gamma_A V_j\rightarrow V_j,\,\,\,a\otimes v\mapsto\chi(a)\cdot v
\]is surjective for all $j\in J$ and $\chi\in \Gamma_A$, then the induced action of $G$ on $\Gamma_A$ is free. Interpreting each $\Gamma_A V_j$ as a (possibly singular) vector bundle over $\Gamma_A/G$, this result means that the induced action of $G$ on $\Gamma_A$ is free if and only if every fibre of each $\Gamma_A V_j$ is ``full", i.e., isomorphic to $V_j$. In the case of compact abelian group actions, the generalized spaces of sections associated to the dual group (which separates the points) are exactly the corresponding isotypic components and the surjectivity condition is, for example, fulfilled if each isotypic component contains an invertible element; a requirement which is in the spirit of actions having ``large" isotypic components (cf. [Pa81]) and leads to a natural concept of \emph{trivial noncommutative principal bundles} (cf. [Wa11b]).

\noindent
We now give a rough outline of the results that can be found in this paper, without going too much into detail:

\subsection*{Outline}

\noindent
A dynamical system $(A,G,\alpha)$ is called smooth if $G$ is a Lie group and the group homomorphism $\alpha:G\rightarrow\Aut(A)$ induces a smooth action of $G$ on $A$. The goal of Section 2 is to show that smooth group actions may also be studied from the viewpoint of smooth dynamical systems, i.e., that each smooth group action induces in a natural way a smooth dynamical system and vice versa.

\noindent
In Section 3 we introduce the concept of a \emph{free dynamical system}. Loosely speaking, we call a dynamical system $(A,G,\alpha)$ free, if the unital locally convex algebra $A$ is commutative and the topological group $G$ admits a family $(\pi_j,V_j)_{j\in J}$ of point separating representations of $G$ such that the evaluation maps defined on the ``generalized spaces of sections" $\Gamma_A V_j:=(A\otimes V_j)^G$ are surjective onto $V_j$ (evaluation with respect to elements of $\Gamma_A$) . We will in particular see how this condition implies that the induced action 
\[\sigma:\Gamma_A\times G\rightarrow\Gamma_A,\,\,\,(\chi,g)\mapsto\chi\circ\alpha(g)
\]of $G$ on the spectrum $\Gamma_A$ of $A$ is free. 



\noindent
In Section 4 we apply the results of Section 3 to dynamical systems arising from group actions in classical geometry. In particular, we will see how this leads to a new characterization of free group actions. For this purpose we have to restrict our attention to Lie groups that admit a family of finite-dimensional continuous point separating representations.
\vspace*{0,2cm}

\noindent
{\bf Theorem} {\bf(Characterization of free group actions).}\index{Characterization of Free Group Actions}
Let $P$ be a manifold, $G$ a compact Lie group and $(C^{\infty}(P),G,\alpha)$ a smooth dynamical system. Then the following statements are equivalent:
\begin{itemize}
\item[(a)]
The smooth dynamical system $(C^{\infty}(P),G,\alpha)$ is free.
\item[(b)]
The induced smooth group action $\sigma:P\times G\rightarrow P$ is free.
\end{itemize}
In particular, in this situation the two concepts of freeness coincide.

\noindent
From a geometrical point of view the previous theorem means that it is possible to test the freeness of a (smooth) group action $\sigma:P\times G\rightarrow P$ in terms of surjective maps defined on spaces of sections of associated (singular) vector bundles.

\noindent
Section 5 is devoted to a more careful discussion of free dynamical systems with compact abelian transformation groups. In particular, we present natural conditions including the corresponding isotypic components which ensure the freeness of such a dynamical system. These conditions do not depend on the commutativity of the algebra $A$ and may therefore be transferred to the context of Noncommutative Geometry (cf. Section \ref{outlook}).

\noindent
In Section 6 we introduce a stronger version of freeness for dynamical systems than the one given in Section \ref{section:free dynamical systems}. In fact, instead of considering arbitrary families $(\pi_j,V_j)_{j\in J}$ of (continuous) point separating representations of a topological group $G$, we restrict our attention to families $(\pi_j,\mathcal{H}_j)_{j\in J}$ of \emph{
unitary irreducible} point separating representations. At this point, we recall that each locally compact group $G$ admits a family of continuous unitary irreducible point separating representations (cf. Theorem \ref{gelfand-raikov}). 

\noindent
The goal of Section 7 is to study some topological aspects of (free) dynamical systems. In particular, we provide conditions which ensure that a dynamical system induces a topological principal bundle. 

\noindent
We finally present some further applications: Section 8 is dedicated to an open problem and an application to the structure theory of $C^*$-algebras. In Section 9 we provide an application to the noncommutative geometry of principal bundles. In fact, we give a short insight how the ``new characterization of free actions" leads to a reasonable definition of (\emph{trivial}) \emph{noncommutative principal bundles} with compact abelian structure group.

\noindent
In the appendix we discuss rudiments on the smooth exponential law the spectrum of the algebra of smooth functions on a manifold.

\subsection*{Preliminaries and Notations} All manifolds appearing in this paper are assumed to be finite-dimensional, paracompact, second countable and smooth. For the necessary background on principal bundles we refer to [KoNo63] or [Hu75]. All algebras are assumed to be complex if not mentioned otherwise. If $A$ is an algebra, we write 
\[\Gamma_A:=\Hom_{\text{alg}}(A,\mathbb{C})\backslash\{0\}
\](with the topology of pointwise convergence on $A$) for the spectrum of $A$. Moreover, a dynamical system is a triple $(A,G,\alpha)$, consisting of a unital locally convex algebra $A$, a topological group $G$ and a group homomorphism $\alpha:G\rightarrow\Aut(A)$, which induces a continuous action of $G$ on $A$.

\subsection*{Acknowledgment} 
We thank Henrik Sepp\"anen for useful hints while writing this paper. In addition we thank the \emph{Studienstiftung des deutschen Volkes} for a doctoral scholarship for my work. 


\section{Dynamical Systems in Classical Differential Geometry}\label{NGENCG}

Since the \emph{Erlanger Programm} of Felix Klein, the defining concept in the study of a geometry has been its symmetry group. In classical differential geometry the symmetries of a manifold are measured by Lie groups, i.e., one studies smooth group actions of a Lie group $G$ acting by diffeomorphisms on a manifold $M$. The goal of this section is to show that smooth group actions may also be studied from the viewpoint of (smooth) dynamical systems, i.e., that each smooth group action induces in a natural way a (smooth) dynamical system and vice versa. We start with the following proposition:

\begin{proposition}\label{smoothness of the group action on the algebra of smooth functions}
If $\sigma:M\times G\rightarrow M$ is a smooth \emph{(}right-\emph{)} action of a Lie group $G$ on a finite-dimensional manifold $M$ \emph{(}possibly with boundary\emph{)} and $E$ is a locally convex space, then the induced \emph{(}left-\emph{)} action 
\[\alpha:G\times C^{\infty}(M,E)\rightarrow C^{\infty}(M,E),\,\,\,\alpha(g,f)(m):=(g.f)(m):=f(\sigma(m,g))
\]of $G$ on the locally convex space $C^{\infty}(M,E)$ is smooth.
\end{proposition}

\begin{proof}
\,\,\,We first recall from [NeWa07], Proposition I.2 that the evaluation map
\[\ev_M:C^{\infty}(M,E)\times M\rightarrow E,\,\,\,(f,m)\mapsto f(m)
\]is smooth. Next, Lemma \ref{smooth exp law} implies that the action map $\alpha$ is smooth if and only if the map
\[\alpha^{\wedge}:C^{\infty}(M,E)\times M\times G\rightarrow E,\,\,\,(f,m,g)\mapsto f(\sigma(m,g))
\]is smooth. Since 
\[\alpha^{\wedge}=\ev_M\circ(\id_{C^{\infty}(M,E)}\times\sigma),
\]we conclude that $\alpha^{\wedge}$ is smooth as a composition of smooth maps. 
\end{proof}

\pagebreak

The previous proposition immediately leads us to the following definition:

\begin{definition}\label{triple}{\bf(Smooth dynamical systems).}\index{Dynamical Systems!Smooth}
We call a dynamical system $(A,G,\alpha)$ \emph{smooth} if $G$ is a Lie group and the group homomorphism $\alpha:G\rightarrow\Aut(A)$ induces a smooth action of $G$ on $A$.
\end{definition}

\begin{example}\label{induced transformation triples}{\bf(Classical group actions).}\sindex[n]{$(C^{\infty}(P),G,\alpha)$}
As the previous discussion shows, a classical example of such a smooth dynamical system is induced by a smooth action $\sigma:M\times G\rightarrow M$ of a Lie group $G$ on a manifold $M$. In particular, each principal bundle $(P,M,G,q,\sigma)$ induces
a smooth dynamical system $(C^{\infty}(P),G,\alpha)$, consisting of the Fr\'echet algebra of smooth functions on the total space $P$, the structure group $G$ and a group homomorphism $\alpha:G\rightarrow\Aut(C^{\infty}(P))$, induced by the smooth action $\sigma:P\times G\rightarrow P$ of $G$ on $P$. For further examples of smooth dynamical systems we refer the interested reader to [Wa11a] (cf. Section \ref{outlook}).
\end{example}

The following proposition characterizes the fixed point algebra of a smooth dynamical system, which is induced from a principal bundle, as the algebra of smooth functions on the corresponding base space:

\begin{proposition}\label{fixed point algebra of principal bundles}
Let $(P,M,G,q,\sigma)$ be a principal bundle and let $(C^{\infty}(P),G,\alpha)$ be the induced smooth dynamical system. 
Then the map
\[\Psi:C^{\infty}(P)^G\rightarrow C^{\infty}(M)\,\,\,\text{defined by}\,\,\,\Psi(f)(q(p)):=f(p)
\]is an isomorphism of Fr\'{e}chet algebras.
\end{proposition}

\begin{proof}
\,\,\,First we observe that the map $\Psi$ is well-defined and a homomorphism of algebras. Further, the universal property of submersions implies that $\Psi(f)$ defines a smooth function on $M$.  

Next, if $\Psi(f)=0$, then the $G$-invariance of $f$ implies that $f=0$. Hence, $\Psi$ is injective. To see that $\Psi$ is surjective, we choose $h\in C^{\infty}(M)$ and put $f:=h\circ q$. Then $f\in C^{\infty}(P)^G$ and $\Psi(f)=h$. The claim now follows the continuity of $\Psi$ and $\Psi^{-1}=q^*$. 
\end{proof}

In the following we will show that if $M$ is a manifold, then each smooth dynamical system of the form $(C^{\infty}(M),G,\alpha)$ induces a smooth action of the Lie group $G$ on $M$. As a first step we endow $\Gamma_{C^{\infty}(M)}$ with the structure of a smooth manifold:

\begin{lemma}\label{spec as manifold}
If $M$ is a manifold, then there is a unique smooth structure on $\Gamma_{C^{\infty}(M)}$ for which the map
\[\Phi:M\rightarrow \Gamma_{C^{\infty}(M)},\,\,\,m\mapsto\delta_m
\]becomes a diffeomorphism.
\end{lemma}

\begin{proof}
\,\,\,Proposition \ref{spec of C(M) top} implies that the map $\Phi$ is a homeomorphism. Therefore, $\Phi$ induces a unique smooth structure on $\Gamma_{C^{\infty}(M)}$ such that $\Phi$ becomes a diffeomorphism. 
\end{proof}

The following observation is well-known, but by a lack of a reference, we give the proof:

\begin{lemma}\label{characterization of smooth maps}
A continuous map $f:M\rightarrow N$ between manifolds $M$ and $N$ is smooth if and only if the composition $g\circ f:M\rightarrow\mathbb{R}$ is smooth for each $g\in C^{\infty}(N,\mathbb{R})$.
\end{lemma}

\begin{proof}
\,\,\,The `` if"-direction is clear. The proof of the other direction is divided into three parts:

(i)  We first note that the map $f$ is smooth if and only if for each $m\in M$ there is an open $m$-neighbourhood $U$ such that $f_{\mid U}:U\rightarrow M$ is smooth. Therefore, let $m\in M$, $n:=f(m)$ and $(\psi,V)$ be a chart around $n$. We now choose an open $n$-neighbourhood $W$ such that $\overline{W}\subseteq V$ and a smooth function $h:N\rightarrow\mathbb{R}$ satisfying $h_{\mid \overline{W}}=1$ and $\supp(h)\subseteq V$. We further choose an open $m$-neighbourhood $U$ such that $f(U)\subseteq W$ (here, we use the continuity of the map $f$). Since the inclusion map $i:W\rightarrow N$ is smooth, it remains to prove that $f_{\mid U}:U\rightarrow W$ is smooth.

(ii) A short observation shows that the map $f_{\mid U}:U\rightarrow W$ is smooth if and only if the map $\psi\circ f_{\mid U}:U\rightarrow\mathbb{R}^n$ is smooth. If $\psi=(\psi_1,\ldots,\psi_n)$, then the last function is smooth if and only if each of its coordinate functions $\psi_i\circ f_{\mid U}:U\rightarrow\mathbb{R}$ is smooth. 

(iii) For fixed $i\in\{1,\ldots,n\}$ we now show that the coordinate function $\psi_i\circ f_{\mid U}:U\rightarrow\mathbb{R}$ is smooth. For this we note that $h_i:=h\cdot\psi_i$ defines a smooth $\mathbb{R}$-valued function on $N$ satisfying ${h_i}_{\mid W}=\psi_i$. Hence, the assumption implies that the map $h_i\circ f:M\rightarrow\mathbb{R}$ is smooth. Since the restriction of a smooth map to an open subsets is smooth again, we conclude from $f(U)\subseteq W$ that 
\[(h_i\circ f)_{\mid U}=\psi_i\circ f_{\mid U}
\]is smooth as desired. This proves the lemma.
\end{proof}


\begin{proposition}\label{smoothness of the group action on the set of characters}
If $M$ is a manifold, $G$ a Lie group and $(C^{\infty}(M),G,\alpha)$ a smooth dynamical system, then the homomorphism $\alpha:G\rightarrow\Aut(C^{\infty}(M))$ induces a smooth \emph{(}right-\emph{)} action
\begin{align*}
\sigma:M\times G\rightarrow M,\,\,\,(\delta_m,g)\mapsto\delta_m\circ\alpha(g)
\end{align*}
of the Lie group $G$ on the manifold $P$. Here, we have identified $M$ with the set of characters via the map $\Phi$ from Lemma \ref{spec as manifold}.
\end{proposition}

\begin{proof}
\,\,\,The proof of this proposition is divided into two parts:

(i) As a first step we again use [NeWa07], Proposition I.2, which states that the evaluation map
\[\ev_M:C^{\infty}(M)\times M\rightarrow \mathbb{K},\,\,\,(f,m)\mapsto f(m).
\]is smooth. From this we conclude that the map $\sigma$ is continuous (cf. Proposition \ref{cont. action II}). 

(ii) In view of part (i), we may use Lemma \ref{characterization of smooth maps} to verify the smoothness of $\sigma$. Indeed, the map $\sigma$ is smooth if and only if the map
\[\sigma_f:M\times G\rightarrow\mathbb{R},\,\,\,(\delta_m,g)\mapsto\sigma(\delta_m,g)(f)=(\alpha(g,f))(m)
\]is smooth for each $f\in C^{\infty}(M,\mathbb{R})$. Therefore, we fix $f\in C^{\infty}(M,\mathbb{R})$ and note that we can write  $\sigma_f$ as $\ev_M\circ(\id_M\times\alpha_f)$, where $\alpha_f:G\rightarrow C^{\infty}(M)\,\,\,g\mapsto\alpha(g,f)$ denotes the smooth orbit map of $f$. Hence, the map $\sigma_f$ is smooth as a composition of smooth maps. Since $f$ was arbitrary, the map $\sigma$ is smooth.
\end{proof}

\begin{remark}\label{inverse constructions}{\bf(Inverse constructions).}
Note that the constructions of Proposition \ref{smoothness of the group action on the algebra of smooth functions} and Proposition \ref{smoothness of the group action on the set of characters} are inverse to each other.
\end{remark}

\begin{remark}\label{remark of free action in geometry}{\bf(Principal bundles).}
Since we are in particularly interested in principal bundles, it is reasonable to ask if there exist natural (algebraic) conditions on a smooth dynamical system $(C^{\infty}(P),G,\alpha)$ which ensure the freeness of the induced action $\sigma$ of $G$ on $P$ of Proposition \ref{smoothness of the group action on the set of characters}. In fact, if this is the case and if the action is additionally proper, then we obtain a principal bundle $(P,P/G,G,\pr,\sigma)$, where $\pr:P\rightarrow P/G,\,\,\,p\mapsto p.G$ denotes the corresponding orbit map. We will treat this question in the next section.
\end{remark}

\section{Free Dynamical Systems}\label{section:free dynamical systems}

In this section we introduce the concept of a \emph{free dynamical system}. Loosely speaking, we call a dynamical system $(A,G,\alpha)$ free, if the unital locally convex algebra $A$ is commutative and the topological group $G$ admits a family $(\pi_j,V_j)_{j\in J}$ of point separating representations of $G$ such that the evaluation maps defined on the ``generalized spaces of sections" $\Gamma_A V_j:=(A\otimes V_j)^G$ are surjective onto $V_j$ (evaluation with respect to elements of $\Gamma_A$) . We will in particular see how this condition implies that the induced action 
\[\sigma:\Gamma_A\times G\rightarrow\Gamma_A,\,\,\,(\chi,g)\mapsto\chi\circ\alpha(g)
\]of $G$ on the spectrum $\Gamma_A$ of $A$ is free. We start with some basics from the representation theory of (topological) groups, which will later be important for deducing the freeness property:

\begin{definition}\label{sep. the. points of G}{\bf(Separating representations).}\index{Representations!Separating}
Let $G$ be topological group. We say that a family $(\pi_j,V_j)_{j\in J}$ of (continuous) representations of $G$ \emph{separates the points of} $G$ if for each $g\in G$ with $g\neq 1_G$, there is a $j\in J$ such that $\pi_j(g)\neq\id_{V_j}$.
\end{definition}

\begin{lemma}\label{princ. bdl. cond.}
Let $G$ be a topological group and suppose that $(\pi_j,V_i)_{j\in J}$ is a family of point separating representations of $G$. If $g\in G$ is such that $\pi_j(g)=\id_{V_j}$ for all $j\in J$, then $g=1_G$.
\end{lemma}

\begin{proof}
\,\,\,The claim immediately follows from Definition \ref{sep. the. points of G}.
\end{proof}

\begin{remark}\label{faithful representations}{\bf(Faithful representations).}\index{Representations!Faithful}
We recall that each faithful representation $(\pi,V)$ of a topological group $G$ separates the points of $G$.
\end{remark}



An important class of groups that admit a family of separating representations is given by the locally compact groups:

\begin{theorem}\label{gelfand-raikov}{\bf(Gelfand--Raikov).}\index{Theorem!of Gelfand--Raikov}
Each locally compact group $G$ admits a family of continuous unitary irreducible representations that separates the points of $G$.
\end{theorem}

\begin{proof}
\,\,\,A proof of this statement can be found in the very nice paper [Yo49].
\end{proof}

\begin{definition}\label{sections again}{\bf(``Associated space of sections'').}\sindex[n]{$\Gamma_A V$}
Let $A$ be a unital locally convex algebra and $G$ a topological group. If $(A,G,\alpha)$ is a dynamical system and $(\pi,V)$ a (continuous) representation of $G$, then there is a natural (continuous) action of $G$ on the tensor product $A\otimes V$ defined on simple tensors by $g.(a\otimes v):=(\alpha(g).a)\otimes(\pi(g).v)$. We write
\begin{align*}
\Gamma_A V:=(A\otimes V)^G=\big\{s\in A\otimes V:(\forall g\in G)\,(\alpha(g)\otimes\id_V)(s)=(\id_A\otimes\pi(g^{-1}))(s)\big\}
\end{align*}
for the set of fixed elements under this action.
\end{definition}

\begin{lemma}
Let $(A,G,\alpha)$ be as in Definition \ref{sections again}. 
If $A^G$ is the corresponding fixed point algebra and $(\pi,V)$ a continuous representation of $G$, then the map
\[\rho:\Gamma_A V\times A^G\rightarrow\Gamma_A V,\,\,\,(a\otimes v,b)\mapsto ab\otimes v
\]defines on $\Gamma_A V$ the structure of a locally convex $A^G$-module.
\end{lemma}

\begin{proof}
\,\,\,According to [Wa11a], Proposition D.2.5, $A\otimes V$ carries the structure of a locally convex $A^G$-module. Thus, a short calculation shows that the same holds for the restriction to the (closed) subspace $\Gamma_A V$.
\end{proof}

\begin{remark}\label{sections of an associated vector bundle top}
Let $(P,M,G,q,\sigma)$ be a principal bundle. Further, let $(\pi,V)$ be a finite-dimensional representation of $G$ defining the associated bundle $\mathbb{V}:=P\times_{\pi} V$ over $M$. If we write
\[C^{\infty}(P,V)^G:=\{f:P\rightarrow V:\,(\forall g\in G)\,f(p.g)=\pi(g^{-1}).f(p)\}
\]for the space of equivariant smooth functions, then the map
\[\Psi_{\pi}:C^{\infty}(P,V)^G\rightarrow\Gamma\mathbb{V}\,\,\,\text{defined by}\,\,\,\Psi_{\pi}(f)(q(p)):=[p,f(p)],
\]is a topological isomorphism of $C^{\infty}(M)$-modules. Indeed, a proof of this statement can be found in [Wa11a], Corollary 3.3.7.
\end{remark}

\begin{example}\label{example of section}{\bf(The classical case).}
If $(P,M,G,q,\sigma)$ is a principal bundle, $(C^{\infty}(P),G,\alpha)$ the corresponding smooth dynamical system from Example \ref{induced transformation triples} and $(\pi,V)$ a finite-di\-men\-sio\-nal  representation of $G$, then an easy observation shows that \[C^{\infty}(P)\otimes V\cong C^{\infty}(P,V)
\] (as Fr\'echet spaces) and further that $\Gamma_{C^{\infty}(P)}V=(C^{\infty}(P)\otimes V)^G\cong C^{\infty}(P,V)^G$. In particular, Remark \ref{sections of an associated vector bundle top} implies that $\Gamma_{C^{\infty}(P)}V$ is topologically isomorphic to $\Gamma\mathbb{V}$ as $C^{\infty}(M)$-module.
\end{example}

\begin{remark} 
In view of Example \ref{example of section}, the $A^G$-module $\Gamma_A V$ generalizes the \emph{space of sections} associated to the dynamical system $(A,G,\alpha)$ and the representation $(\pi,V)$ of $G$.
\end{remark}

We now come to the central definition of this section. Note that $A$ is assumed to be a commutative algebra, since our considerations depend on the existence of enough characters:

\begin{definition}\label{free dynamical systems}{\bf (Free dynamical systems).}\index{Dynamical Systems!Free}
Let $A$ be a commutative unital locally convex algebra and $G$ a topological group. A dynamical system $(A,G,\alpha)$ is called \emph{free} if there exists a family $(\pi_j,V_j)_{j\in J}$ of point separating representations of $G$ such that the map
\[\ev^j_{\chi}:=\ev^{V_j}_{\chi}:\Gamma_A V_j\rightarrow V_j,\,\,\,a\otimes v\mapsto\chi(a)\cdot v
\]is surjective for all $j\in J$ and all $\chi\in\Gamma_A$. 
\end{definition}

\begin{theorem}\label{freeness of induced action}{\bf(Freeness of the induced action).}\index{Freeness!of the Induced Action}
If $(A,G,\alpha)$ is a free dynamical system, then the induced action
\[\sigma:\Gamma_A\times G\rightarrow\Gamma_A,\,\,\,(\chi,g)\mapsto\chi\circ\alpha(g)
\]of $G$ on the spectrum $\Gamma_A$ of $A$ is free. 
\end{theorem}

\begin{proof}
\,\,\,We divide the proof of this theorem into four parts:

(i) In order to verify the freeness of the map $\sigma$, we have to show that the stabilizer of each element of $\Gamma_A$ is trivial: Consequently, we fix $\chi_0\in\Gamma_A$ and let $g_0\in G$ with $\chi_0\circ\alpha(g_0)=\chi_0$.

(ii) Since $(A,G,\alpha)$ is assumed to be a free dynamical system, there exists a family $(\pi_j,V_j)_{j\in J}$ of point separating representations of $G$ for which the map
\[\ev^j_{\chi}:\Gamma_A V_j\rightarrow V_j,\,\,\,a\otimes v\mapsto\chi(a)\cdot v
\]is surjective for all $j\in J$ and all $\chi\in\Gamma_A$. In particular, we can choose $j\in J$, $v\in V_j$ and $s\in\Gamma_A V_j$ with $\ev^j_{\chi_0}(s)=v$. We recall that the element $s$ satisfies the equation
\begin{align}
(\alpha(g_0)\otimes\id_{V_j})(s)=(\id_A\otimes\pi_j(g_0^{-1}))(s).\label{freeness equation}
\end{align}

(iii) Applying $\chi_0\otimes\id_{V_j}$ to the left of equation (\ref{freeness equation}) leads to
\[((\chi_0\circ\alpha(g_0))\otimes\id_{V_j})(s)=(\chi_0\otimes\pi_j(g_0^{-1}))(s).
\]Thus, we conclude from $\chi_0\circ\alpha(g_0)=\chi_0$ that
\[(\chi_0\otimes\id_{V_j})(s)=(\chi_0\otimes\pi_j(g_0^{-1}))(s)=\pi_j(g_0^{-1})((\chi_0\otimes\id_{V_j})(s)).
\]

(iv) We finally note that $s\in\Gamma_A V_j$ implies that
\[(\chi_0\otimes\id_{V_j})(s)=\ev^j_{\chi_0}(s)=v.
\]In view of part (iii) this shows that $v=\pi_j(g_0)(v)$. As $j\in J$ and $v\in V_j$ were arbitrary, we conclude that $\pi_j(g_0)=\id_{V_j}$ for all $j\in J$ and therefore that $g_0=1_G$ (cf. Lemma \ref{princ. bdl. cond.}). This completes the proof.
\end{proof}

\section{A New Characterization of Free Group Actions in Classical Geometry}\label{ANCFGACG}

In this section we apply the results of the previous section to dynamical systems arising from group actions in classical geometry. In particular, we will see how this leads to a new characterization of free group actions. For this purpose we have to restrict our attention to Lie groups that admit a family of finite-dimensional continuous point separating representations.

\begin{definition}\label{linearizer}{\bf(The linearizer).}\index{Linearizer}
For a Lie group $G$ we define its \emph{linearizer} as the subgroup $\Lin(G)$\sindex[n]{$\Lin(G)$} which is the intersection of the kernels of all finite-dimensional continuous representations of $G$.
\end{definition}

\begin{lemma}\label{linerarizer/point sep. rep.}
Let $G$ be a Lie group. Then the following statements are equivalent:
\begin{itemize}
\item[\emph{(a)}]
The finite-dimensional continuous representations separate the points of $G$.
\item[\emph{(b)}]
$\Lin(G)=\{1_G\}$.
\end{itemize}
\end{lemma}

\begin{proof}
\,\,\,We just have to note that $\Lin(G)$ is non-trivial if and only if there exists an element $g\in G$ which lies in the kernels of all finite-dimensional continuous representation of $G$.
\end{proof}

The following theorem shows that the property $\Lin(G)=\{1_G\}$ characterizes the groups which admit a faithful finite-dimensional linear representation:

\begin{theorem}\label{linear lie groups}{\bf(Existence of faithful finite-dimensional representations).}\index{Representations!Faithful Finite-Dimensional}
For a connected Lie group $G$ the following statements are equivalent:
\begin{itemize}
\item[\emph{(a)}]
There exists a faithful finite-dimensional continuous representation of $G$.
\item[\emph{(b)}]
There exists a faithful finite-dimensional continuous representation of $G$ with closed image.
\item[\emph{(c)}]
$\Lin(G)=\{1_G\}$.
\end{itemize}
\end{theorem}

\begin{proof}
\,\,\,(a) $\Leftrightarrow$ (c): The proof of this equivalence is part of [HiNe10], Theorem 15.2.7.

(a) $\Leftrightarrow$ (b): This nontrivial statement is carried out as a bunch of exercises at the end of [HiNe10], Section 15.2.
\end{proof}

\begin{remark}
In view of Lemma \ref{linerarizer/point sep. rep.} and Theorem \ref{linear lie groups}, it is exactly the \emph{linear Lie groups} that admit a family of finite-dimensional continuous point separating representations.
\end{remark}

\begin{theorem}\label{free action theorem}
Let $P$ be a manifold and $G$ be a Lie group. Then the following assertions hold:
\begin{itemize}
\item[\emph{(a)}]
If the smooth dynamical system $(C^{\infty}(P),G,\alpha)$ is free and, in addition the induced action $\sigma$ is proper, then we obtain a principal bundle $(P,P/G,G,\pr,\sigma)$ \emph{(}cf. Remark \ref{remark of free action in geometry}\emph{)}.
\item[\emph{(b)}]
Conversely, if $G$ is a linear Lie group and $(P,M,G,q,\sigma)$ a principal bundle, then the corresponding smooth dynamical system $(C^{\infty}(P),G,\alpha)$ is free.
\end{itemize}
\end{theorem}

\begin{proof}
\,\,\,(a) We first recall that the induced action $\sigma:P\times G\rightarrow P$ is smooth by Proposition \ref{smoothness of the group action on the set of characters}. Furthermore, Theorem \ref{freeness of induced action} implies that the map $\sigma$ is free. Since $\sigma$ is additionally assumed to be proper, the claim now follows from the Quotient Theorem (cf. [To00], Kapitel VIII, Abschnitt 21), which states that each free and proper smooth action $\sigma:P\times G\rightarrow P$ defines a principal bundle of the form $(P,P/G,G,\pr,\sigma)$.

(b) 
Since $G$ is a linear Lie group, there exists a faithful finite-dimensional continuous representation $(\pi,V)$ of $G$. In particular, this representation separates the points of $G$ (cf. Remark \ref{faithful representations}). In order to prove the freeness of the smooth dynamical system $(C^{\infty}(P),G,\alpha)$, it would therefore be enough to show that the map
\begin{align}
\ev_{p}:C^{\infty}(P,V)^G\rightarrow V,\,\,\,f \mapsto f(p)\label{check condition}
\end{align}
is surjective for all $p\in P$. We proceed as follows: 

(i) We first observe that the surjectivity of the maps (\ref{check condition}) is a local condition. Further, we (again) recall that according to Remark \ref{sections of an associated vector bundle top} the map
\[\Psi_{\pi}:C^{\infty}(P,V)^G\rightarrow\Gamma\mathbb{V},\,\,\,\Psi_{\pi}(f)(q(p)):=[p,f(p)],
\]where $\mathbb{V}$ denotes the vector bundle over $M$ associated to $(P,M,G,q,\sigma)$ via the representation $(\pi,V)$ of $G$, is a (topological) isomorphism of $C^{\infty}(M)$-modules.

(ii) Now, we choose $p\in P$, $v\in V$ and construct a smooth section $s\in\Gamma\mathbb{V}$ with $s(q(p))=[p,v]$. Indeed, such a section can always be constructed locally and then extended to the whole of $M$ by multiplying with a smooth bump function. The construction of $s$ implies that the function $f_s:=\Psi_{\pi}^{-1}(s)\in C^{\infty}(P,V)^G$ satisfies $f_s(p)=v$. As $p\in P$, $v\in V$ were arbitrary, this completes the proof.
\end{proof}

\begin{remark}
Note that Theorem \ref{free action theorem} (a) means that it is possible to test the freeness of a (smooth) group action $\sigma:P\times G\rightarrow P$ in terms of surjective maps defined on spaces of sections of associated (singular) vector bundles.
\end{remark}

\begin{remark}{\bf(Why linear Lie groups?).}
(a) The crucial idea of the proof of Theorem \ref{free action theorem} (b) is to use the identification of the space $C^{\infty}(P,V)^G$ with the space of sections $\Gamma\mathbb{V}$ of the vector bundle $\mathbb{V}$ over $M$ associated to $(P,M,G,q,\sigma)$ via the representation $(\pi,V)$ of $G$. Note that this identification only holds for smooth representations $(\pi,V)$ of $G$. Since finite-dimensional continuous representations of Lie groups are automatically smooth and there is not many literature about smooth point separating representations of Lie groups, we restrict our attention to Lie groups that admit a family of finite-dimensional continuous point separating representations. In this context, it is an interesting observation that the natural action of an arbitrary Lie group $G$ on $C^{\infty}(G,\mathbb{R})$ separates the points of $G$.
\end{remark}

The following corollary gives a one-to-one correspondence between free dynamical systems and free group action in the case where the structure group $G$ is a compact Lie group:

\begin{corollary}\label{characterization of free group actions}{\bf(Characterization of free group actions).}\index{Characterization of Free Group Actions}
Let $P$ be a manifold, $G$ a compact Lie group and $(C^{\infty}(P),G,\alpha)$ a smooth dynamical system. Then the following statements are equivalent:
\begin{itemize}
\item[\emph{(a)}]
The smooth dynamical system $(C^{\infty}(P),G,\alpha)$ is free.
\item[\emph{(b)}]
The induced smooth group action $\sigma:P\times G\rightarrow P$ is free.
\end{itemize}
In particular, in this situation the two concepts of freeness coincide.
\end{corollary}

\begin{proof}
\,\,\,We first note that each compact Lie group $G$ admits a faithful finite-dimensional continuous representation, i.e., each compact Lie group $G$ is linear. Indeed, a proof of this statement can be found in [HiNe10], Theorem 11.3.9. Moreover, since $G$ is compact, the properness of the action $\sigma$ is automatic. Hence, the equivalence follows from Theorem \ref{free action theorem}. The last statement is now a consequence of Remark \ref{inverse constructions}.
\end{proof}

\section{Free Dynamical Systems with Compact Abelian Structure Group}\label{free dyn sys casg}

In this section we rewrite the freeness condition for a dynamical system $(A,G,\alpha)$ with compact abelian structure group $G$. In particular, we present natural conditions which ensure the freeness of such a dynamical system. These conditions do not depend on the commutativity of the algebra $A$ and may therefore be transferred to the context of Noncommutative Geometry (cf. Section \ref{outlook}). Given a dynamical system $(A,G,\alpha)$ with compact abelian structure group $G$ we write $\widehat{G}$ for the character group of $G$ and 
\[A_{\varphi}:=\{a\in A:\,(\forall g\in G)\,\,\alpha(g).a=\varphi(g)\cdot a\}
\]for the isotypic component corresponding to the character $\varphi\in\widehat{G}$. 

\begin{lemma}\label{A_phi=gammaC}
Let $A$ be a unital locally convex algebra, $G$ a compact abelian group and $(A,G,\alpha)$ a dynamical system. Further let $\varphi:G\rightarrow\mathbb{C^{\times}}$ be a character and $\pi_{\varphi}:G\rightarrow\GL_1(\mathbb{C})$ be the corresponding representation given by $\pi_{\varphi}(g).z=\varphi(g)\cdot z$ for all $z\in\mathbb{C}$. Then the map 
\[\Gamma_A\mathbb{C}\rightarrow A_{\varphi^{-1}},\,\,\,a\otimes 1\mapsto a
\]is an isomorphism of locally convex spaces.
\end{lemma}

\begin{proof}
\,\,\,For the proof we just note that $a\otimes 1\in\Gamma_A\mathbb{C}$ implies $\alpha(g)(a)\otimes 1=\varphi^{-1}(g)\cdot a\otimes 1$.
\end{proof}

\begin{remark}\label{char sep the points}{\bf(The characters separate the points).}
At this stage we recall that the characters of a compact abelian group separate the points. Indeed, this statement is a consequence of Theorem \ref{gelfand-raikov} and Schur's Lemma (cf. [Ne09], Theorem 4.2.7).
\end{remark}

\begin{proposition}\label{freeness for compact abelian groups I}{\bf(The freeness condition for compact abelian groups).}\index{Freeness!Condition for Compact Abelian Groups}
Let $A$ be a commutative unital locally convex algebra and $G$ a compact abelian group. A dynamical system $(A,G,\alpha)$ is free in the sense of Definition \ref{free dynamical systems} if the map
\[\ev^{\varphi}_{\chi}: A_{\varphi}\rightarrow \mathbb{C},\,\,\,a\mapsto\chi(a)
\]is surjective for all $\varphi\in\widehat{G}$ and all $\chi\in\Gamma_A$.
\end{proposition}

\begin{proof}
\,\,\,The claim is a consequence of Proposition \ref{char sep the points} and Lemma \ref{A_phi=gammaC}. 
\end{proof}

\begin{remark}
We recall from [HoMo06], Proposition 2.42 that each compact abelian Lie group $G$ is isomorphic to $\mathbb{T}^n\times \Lambda$ for some natural number $n$ and a finite abelian group $\Lambda$. In particular, the character group $\widehat{G}$ of a compact abelian Lie group is finitely generated.
\end{remark}

\begin{corollary}\label{freeness for compact abelian groups I,5}
Let $A$ be a commutative unital locally convex algebra, $G$ a compact abelian Lie group and $(A,G,\alpha)$ a dynamical system. Further, let $(\varphi_i)_{i\in I}$ be a finite set of generators of $\widehat{G}$. Then the following two conditions are equivalent:
\begin{itemize}
\item[\emph{(a)}]
The map
\[\ev^{\varphi}_{\chi}: A_{\varphi}\rightarrow \mathbb{C},\,\,\,a\mapsto\chi(a)
\]is surjective for all $\varphi\in\widehat{G}$ and all $\chi\in\Gamma_A$.
\item[\emph{(b)}]
The map
\[\ev^{\varphi_i}_{\chi}: A_{\varphi_i}\rightarrow \mathbb{C},\,\,\,a\mapsto\chi(a)
\]is surjective for all $i\in I$ and all $\chi\in\Gamma_A$.
\end{itemize}
In particular, if one of the statements holds, then the dynamical system $(A,G,\alpha)$ is free.
\end{corollary}

\begin{proof}
\,\,\,(a) $\Rightarrow$ (b): This direction is trivial.

(b) $\Rightarrow$ (a): For the second direction we fix $\chi\in\Gamma_A$. Further, we choose for each $i\in I$ an element $a_{\varphi_i}\in A_{\varphi_i}$ with $\chi(a_{\varphi_i})\neq 0$. Now, if $\varphi\in\widehat{G}$, then there exist $k\in\mathbb{N}$ and integers $n_1,\ldots,n_k\in\mathbb{Z}$ such that 
\[\varphi=\varphi_{i_1}^{n_1}\cdots\varphi_{i_k}^{n_k}\,\,\,\text{for some}\,\,i_1,\ldots,i_k\in I.
\]Hence, the element $a_{\varphi}:=a_{\varphi_{i_1}}^{n_1}\cdots a_{\varphi_{i_k}}^{n_k}\in A_{\varphi}$ satisfies $\chi(a_{\varphi})\neq 0$.

The last assertion is a direct consequence of Proposition \ref{freeness for compact abelian groups I}.
\end{proof}

\begin{proposition}\label{freeness for compact abelian groups II}{\bf(Invertible elements in isotypic components).}\index{Invertible Elements in Isotypic Components}
Let $A$ be a commutative unital locally convex algebra, $G$ a compact abelian group and $(A,G,\alpha)$ a dynamical system. If each isotypic component $A_{\varphi}$ contains an invertible element, then the dynamical system $(A,G,\alpha)$ is free.
\end{proposition}

\begin{proof}
\,\,\,The assertion easily follows from Proposition \ref{freeness for compact abelian groups I}. Indeed, if $a_{\varphi}\in A_{\varphi}$ is invertible, then $\chi(a)\neq 0$ for all $\chi\in\Gamma_A$.
\end{proof}

\begin{remark}
Note that 
it is possible to ask for invertible elements in the isotypic components even if the algebra $A$ is noncommutative. We will use this fact in the following chapter.
\end{remark}

%


\begin{proposition}\label{set of gen}
Let $A$ be a unital locally convex algebra, $G$ a compact abelian group and $(A,G,\alpha)$ a dynamical system. Further, let $(\varphi_i)_{i\in I}$ be a finite set of generators of $\widehat{G}$. Then the following two statements are equivalent:
\begin{itemize}
\item[\emph{(a)}]
$A_{\varphi}$ contains invertible elements for all $\varphi\in\widehat{G}$.
\item[\emph{(b)}]
$A_{\varphi_i}$ contains invertible elements for all $i\in I$.
\end{itemize}
In particular, if $A$ is commutative and one of the statements holds, then the dynamical system $(A,G,\alpha)$ is free.
\end{proposition}

\begin{proof}
\,\,\,(a) $\Rightarrow$ (b): This direction is trivial.

(b) $\Rightarrow$ (a): For each $i\in I$ we choose an invertible element $a_{\varphi_i}\in A_{\varphi_i}$. Next, if $\varphi\in\widehat{G}$, then there exist $k\in\mathbb{N}$ and integers $n_1,\ldots,n_k\in\mathbb{Z}$ such that 
\[\varphi=\varphi_{i_1}^{n_1}\cdots\varphi_{i_k}^{n_k}\,\,\,\text{for some}\,\,i_1,\ldots,i_k\in I.
\]Hence, $a_{\varphi}:=a_{\varphi_{i_1}}^{n_1}\cdots a_{\varphi_{i_k}}^{n_k}$ is an invertible element in $A_{\varphi}$.

The last assertion is a direct consequence of Proposition \ref{freeness for compact abelian groups II}.
\end{proof}



The following proposition shows that if all isotypic components of a dynamical system contain invertible elements, then they are ``mutually" isomorphic to each other as modules of the fixed point algebra:

\begin{proposition}\label{iso of A_1-modules}
Let $A$ be a commutative unital locally convex algebra and $G$ a compact abelian group. Further, let $(A,G,\alpha)$ be a dynamical system and $A^G$ the corresponding fixed point algebra. If each isotypic component $A_{\varphi}$ contains an invertible element, then the map
\[\Psi_{\varphi}:A^G\rightarrow A_{\varphi} ,\,\,\,a\mapsto a_{\varphi}a,
\]where $a_{\varphi}$ denotes some fixed invertible element in $A_{\varphi}$, is an isomorphism of locally convex $A^G$-modules for each $\varphi\in\widehat{G}$. In particular, each isotypic component $A_{\varphi}$ is a free $A^G$-module.
\end{proposition}

\begin{proof}
\,\,\,An easy calculation shows that $\Psi_{\varphi}$ is a morphism of locally convex $A^ G$-modules, and therefore the assertion follows from the fact that $a_{\varphi}\in A_{\varphi}$ is invertible. 
\end{proof}

%

We finally apply the results of this section to dynamical systems arising from classical geometry. The following theorem may be viewed as a first answer to Remark \ref{remark of free action in geometry}. A more detailed analysis can be found in Section \ref{outlook} or in [Wa11a].

\begin{theorem}\label{free action for triples}
Let $P$ be a manifold and $G$ be a compact abelian Lie group. Further, let $(C^{\infty}(P),G,\alpha)$ be a smooth dynamical system. If $\pr:P\rightarrow P/G$ denotes the orbit map corresponding to the action of $G$ on $P$ \emph{(}cf. Proposition \ref{smoothness of the group action on the set of characters}\emph{)}, then the following assertions hold:
\begin{itemize}
\item[\emph{(a)}]
If each isotypic component $C^{\infty}(P)_{\varphi}$ contains an invertible element, then we obtain a principal bundle $(P,P/G,G,\pr,\sigma)$.
\item[\emph{(b)}]
If $(\varphi_i)_{i\in I}$ is a finite set of generators of $\widehat{G}$ and each subspace $C^{\infty}(P)_{\varphi_i}$ contains an invertible element, then we obtain a principal bundle $(P,P/G,G,\pr,\sigma)$.
\end{itemize}
\end{theorem}

\begin{proof}
\,\,\,(a) Since $G$ is compact, the induced action $\sigma$ is automatically proper. Therefore, the first assertion follows from Theorem \ref{free action theorem} and Proposition \ref{freeness for compact abelian groups II}.

(b) The second assertion follows from Proposition \ref{set of gen} (b) and part (a).
\end{proof}

\section{Strongly Free Dynamical Systems}

We introduce a stronger version of freeness for dynamical systems than the one given in Section \ref{section:free dynamical systems} (cf. Definition \ref{free dynamical systems}). In fact, instead of considering arbitrary families $(\pi_j,V_j)_{j\in J}$ of (continuous) point separating representations of a topological group $G$, we restrict our attention to families $(\pi_j,\mathcal{H}_j)_{j\in J}$ of \emph{
unitary irreducible} point separating representations. At this point, we recall that each locally compact group $G$ admits a family of continuous unitary irreducible point separating representations (cf. Theorem \ref{gelfand-raikov}). We show that Theorem \ref{freeness of induced action} and Theorem \ref{characterization of free group actions} stay true in this context of \emph{strongly free} dynamical systems and that Proposition \ref{freeness for compact abelian groups I} actually turns into a definition for strongly free dynamical systems with compact abelian structure group. We close this short section with a nice example.

\begin{definition}\label{free dynamical systems,strong}{\bf (Strongly free dynamical systems).}\index{Dynamical Systems!Strongly Free}
Let $A$ be a commutative unital locally convex algebra and $G$ a topological group. A dynamical system $(A,G,\alpha)$ is called \emph{strongly free} if there exists a family $(\pi_j,\mathcal{H}_j)_{j\in J}$ of unitary irreducible point separating representations of $G$ such that the map
\[\ev^j_{\chi}:=\ev^{\mathcal{H}_j}_{\chi}:\Gamma_A \mathcal{H}_j\rightarrow \mathcal{H}_j,\,\,\,a\otimes v\mapsto\chi(a)\cdot v
\]is surjective for all $j\in J$ and all $\chi\in\Gamma_A$. 
\end{definition}

\begin{proposition}\label{freeness of induced action,strong}{\bf(Freeness of the induced action again).}\index{Freeness!of the Induced Action}
If $(A,G,\alpha)$ is a strongly free dynamical system, then the induced action
\[\sigma:\Gamma_A\times G\rightarrow\Gamma_A,\,\,\,(\chi,g)\mapsto\chi\circ\alpha(g)
\]of $G$ on the spectrum $\Gamma_A$ of $A$ is free. 
\end{proposition}

\begin{proof}
\,\,\,This assertion immediately follows from Theorem \ref{freeness of induced action}, since each strongly free dynamical system is free.
\end{proof}

\begin{proposition}\label{characterization of free group actions,strong}{\bf(Characterization of free group actions again).}\index{Characterization of Free Group Actions}
Let $P$ be a manifold, $G$ a compact Lie group and $(C^{\infty}(P),G,\alpha)$ a smooth dynamical system. Then the following statements are equivalent:
\begin{itemize}
\item[\emph{(a)}]
The smooth dynamical system $(C^{\infty}(P),G,\alpha)$ is strongly free.
\item[\emph{(b)}]
The induced smooth group action $\sigma:P\times G\rightarrow P$ is free.
\end{itemize}
In particular, in this situation the concepts of freeness coincide.
\end{proposition}

\begin{proof}
\,\,\,This assertion can be proved similarly to Corollary \ref{characterization of free group actions} (cf. Theorem \ref{free action theorem}): In fact, given a finite-dimensional representation $(\pi,V)$ of $G$, we just have to note that it is possible to find an inner product on $V$ such that $G$ acts by unitary transformations (``Weyl's trick") and that each unitary finite-dimensional representation of $G$ can be decomposed into the (finite) sum of irreducible representations. 
\end{proof}


\begin{proposition}\label{freeness for compact abelian groups,strong}{\bf(The strong freeness condition for compact abelian groups).}\index{Strong Freeness!Condition for Compact Abelian Groups}
Let $A$ be a commutative unital locally convex algebra and $G$ a compact abelian group. A dynamical system $(A,G,\alpha)$ is strongly free in the sense of Definition \ref{free dynamical systems,strong} if and only if the map
\[\ev^{\varphi}_{\chi}: A_{\varphi}\rightarrow \mathbb{C},\,\,\,a\mapsto\chi(a)
\]is surjective for all $\varphi\in\widehat{G}$ and all $\chi\in\Gamma_A$.
\end{proposition}

\begin{proof}
\,\,\,($``\Leftarrow"$) This direction is obvious, since the characters of the group $G$ induce a family of unitary irreducible representations that separate the points of $G$ (cf. Lemma \ref{A_phi=gammaC} and Proposition \ref{char sep the points}).

($``\Rightarrow"$) For the other direction we first note that each unitary irreducible representation $(\pi,\mathcal{H})$ of $G$ is one-dimensional by Schur's Lemma (cf. [Ne09], Theorem 4.2.7), i.e., $\pi(g).v=\varphi(g)\cdot v$ for all $g\in G$, $v\in\mathcal{H}$ and some character $\varphi$ of $\widehat{G}$. Thus, if the dynamical system $(A,G,\alpha)$ is strongly free and $(\pi_j,\mathcal{H}_j)_{j\in J}$ is a family of unitary irreducible point separating representations of $G$ satisfying the conditions of Definition \ref{free dynamical systems,strong}, then [HoMo06], Corollary 2.3.3. (i) implies that the corresponding characters $\varphi_j$ generate $\widehat{G}$ and from this we easily conclude that the map
\[\ev^{\varphi}_{\chi}: A_{\varphi}\rightarrow \mathbb{C},\,\,\,a\mapsto\chi(a)
\]is surjective for all $\varphi\in\widehat{G}$ and all $\chi\in\Gamma_A$ (cf. Lemma \ref{A_phi=gammaC}).
\end{proof}

\begin{example}
We now want to use Proposition \ref{freeness for compact abelian groups,strong} to show that the action of the group $C_2:=\{-1,+1\}$ on $\mathbb{R}$ defined by
\[\sigma:\mathbb{R}\times C_2\rightarrow\mathbb{R},\,\,\,r.(-1):=\sigma(r,-1):=-r
\]is not free: Indeed, we first note that the map
\[\Psi:C_2\rightarrow\Hom_{\text{gr}}(C_2,\mathbb{T}),\,\,\,\Psi(-1)(-1):=-1
\]is an isomorphism of abelian groups. From this we easily conclude that the isotypic component of the associated smooth dynamical system $(C^{\infty}(\mathbb{R}),C_2,\alpha)$ (cf. Proposition \ref{smoothness of the group action on the algebra of smooth functions})  corresponding to the generator $-1\in C_2$ is given by
\[C^{\infty}(\mathbb{R})_{-1}=\{f:\mathbb{R}\rightarrow\mathbb{C}:\,(\forall r\in\mathbb{R})\,f(-r)=-f(r)\}.
\]
Since $f(0)=0$ for each $f\in C^{\infty}(\mathbb{R})_{-1}$, the map
\[\ev^{-1}_{0}: C^{\infty}(\mathbb{R})_{-1}\rightarrow \mathbb{C},\,\,\,f\mapsto f(0)=0
\]is not surjective showing that $(C^{\infty}(\mathbb{R}),C_2,\alpha)$ is not strongly free (cf. Proposition \ref{freeness for compact abelian groups,strong}). Therefore, Proposition \ref{characterization of free group actions,strong} implies that the action $\sigma$ is not free.
\end{example}

\section{Some Topological Aspects of Free Dynamical Systems}

In this section we discuss some topological aspects of (free) dynamical systems. Our main goal is to provide conditions which ensure that a dynamical system induces a topological principal bundle. Again, all groups are assumed to act continuously by morphisms of algebras. We start with the following lemma:

\begin{lemma}\label{cont. action I}{\bf(Continuity of the evaluation map).}\index{Continuity!of the Evaluation Map}
Let $A$ be a commutative unital locally convex algebra. If $\Gamma_A$ is locally equicontinuous, then the evaluation map
\[\ev_A:\Gamma_A\times A\rightarrow\mathbb{C},\,\,\,(\chi,a)\mapsto\chi(a)
\]is continuous.
\end{lemma}

\begin{proof}
\,\,\,To prove the continuity of $\ev_A$, we pick $(\chi_0,a_0)\in \Gamma_A\times A$, $\epsilon>0$ and choose an equicontinuous neighbourhood $V$ of $\chi_0$ in $\Gamma_A$ such that
\[V\subseteq\big\{\chi\in\Gamma_A:\,\vert(\chi-\chi_0)(a_0)\vert<\frac{\epsilon}{2}\big\}.
\]Further, we choose a neighbourhood $W$ of $a_0$ in $A$ such that $\vert\chi(a-a_0)\vert<\frac{\epsilon}{2}$ for all $a\in W$ and $\chi\in V$. We thus obtain
\[\vert \chi(a)-\chi_0(a_0)\vert\leq\vert \chi(a)-\chi(a_0)\vert+\vert \chi(a_0)-\chi_0(a_0)\vert\leq\frac{\epsilon}{2}+\frac{\epsilon}{2}=\epsilon
\]for all $\chi\in V$ and $a\in W$.
\end{proof}

\begin{remark}{\bf(Sources of algebras with equicontinuous spectrum).}\label{equicont}
(a) A unital locally convex algebra $A$ is called a \emph{continuous inverse algebra}, or CIA for short, if its group of units $A^{\times}$ is open in $A$ and the inversion 
\[\iota:A^{\times}\rightarrow A^{\times},\,\,\,a\mapsto a^{-1}
\]is continuous at $1_A$. The spectrum $\Gamma_A$ of each CIA $A$ is equicontinuous. In fact, let $U$ be a balanced $0$-neighbourhood such that $U\subseteq 1_A-A^{\times}$. Then $\vert\Gamma_A(U)\vert<1$ (cf. [Wa11a], Appendix C).

(b) Moreover, if $A$ is a \emph{$\rho$-seminormed} algebra, then [Ba00], Corollary 7.3.9 implies that $\Gamma^{\text{cont}}_A$ is equicontinuous.\index{Algebra! Seminormed}
\end{remark}

\begin{proposition}\label{cont. action II}{\bf(Continuity of the induced action map).}\index{Continuity!of the Induced Action Map}
Let $A$ be a commutative unital locally convex algebra, $G$ a topological group and $(A,G,\alpha)$ a dynamical system. If the evaluation map $\ev_A$ is continuous, then the induced action
\[\sigma:\Gamma_A\times G\rightarrow\Gamma_A,\,\,\,(\chi,g)\mapsto\chi\circ\alpha(g)
\]of $G$ on $\Gamma_A$ is continuous.
\end{proposition}

\begin{proof}
\,\,\,The topology (of pointwise convergence) on $\Gamma_A$ implies that the map $\sigma$ is continuous if and only if the maps
\[\sigma_a:\Gamma_A\times G\rightarrow\mathbb{C},\,\,\,(\chi,g)\mapsto\chi(\alpha(g)(a))
\]are continuous for all $a\in A$. Therefore, we fix $a\in A$ and note that $\sigma_a=\ev_A\circ(\id_{\Gamma_A}\times\alpha_a)$,
where 
\[\alpha_a:G\rightarrow A,\,\,\,g\mapsto\alpha(g,a)
\]denotes the continuous orbit map of $a$. In view of the assumption, the map $\sigma_a$ is continuous as a composition of continuous maps. Since $a$ was arbitrary, this proves the proposition.
\end{proof}

\begin{remark}\label{orbit map of group action}
Recall that if $\sigma:X\times G\rightarrow X$ is an action of a topological group $G$ on a topological space $X$, then the orbit map $\pr:X\rightarrow X/G$, $x\mapsto x.G:=\sigma(x,G)$ is surjective, continuous and open. 
\end{remark}

\begin{proposition}\label{cont. action III}
Let $A$ be a commutative unital locally convex algebra, $G$ a compact group and $(A,G,\alpha)$ a dynamical system. If the induced action $\sigma:\Gamma_A\times G\rightarrow\Gamma_A$ is free and continuous, then the following assertions hold:
\begin{itemize}
\item[\emph{(a)}]
For each $\chi\in\Gamma_A$ the map $\sigma_{\chi}:G\rightarrow\Gamma_A,\,\,\,g\mapsto\chi.g:=\sigma(\chi,g)$ is a homeomorphism of $G$ onto the orbit $\mathcal{O}_{\chi}$.
\item[\emph{(b)}]
If $\Gamma_A$ is locally compact, then the orbit space $\Gamma_A/G$ is locally compact and Hausdorff.
\item[\emph{(c)}]
For each pair $(\chi,\chi')\in\Gamma_A\times\Gamma_A$ with $\mathcal{O}_{\chi}=\mathcal{O}_{\chi'}$ there is a unique $\tau(\chi,\chi')\in G$ such that $\chi.\tau(\chi,\chi')=\chi'$, and the map
\[\tau:\Gamma_A\times_{\Gamma_A/G}\Gamma_A:=\{(\chi,\chi')\in\Gamma_A\times\Gamma_A:\,\pr(\chi)=\pr(\chi')\}\rightarrow G
\]is continuous and surjective.
\end{itemize}
\end{proposition}
\begin{proof}
\,\,\,(a) The map $\sigma_{\chi}$ is continuous because $\sigma$. Further, the bijectivitiy of $\sigma_{\chi}$ follows from the freeness of $\sigma$. Since $G$ is compact and $\Gamma_A$ is Hausdorff, a well-known Theorem from topology now implies that $\sigma_{\chi}$ is a homeomorphism of $G$ onto the orbit $\mathcal{O}_{\chi}$.

(b) If $\Gamma_A$ is locally compact, then the orbit space $\Gamma_A/G$ is locally compact because the orbit map is open and continuous. Moreover, the compactness of $G$ implies that the action $\sigma$ is proper. Therefore, the image of the map
\[\Gamma_A\times G\rightarrow\Gamma_A\times\Gamma_A,\,\,\,(\chi,g)\mapsto(\chi,\chi.g)
\]is a closed subset of $\Gamma_A\times\Gamma_A$. Now, the assertion follows from Remark \ref{orbit map of group action} and the more general fact that the target space of a surjective, continuous, open map $f:X\rightarrow Y$ is Hausdorff if and only if the preimage of the diagonal under $f\times f$ is closed. 

(c) Suppose $\chi_i\to\chi$, $\chi'_i\to\chi'$, and $\mathcal{O}_{\chi}=\mathcal{O}_{\chi'}$ so that by definition, $\chi_i.\tau(\chi_i,\chi'_i)=\chi'_i$. Since $G$ is compact, we can assume by passing to a subnet that $\tau(\chi_i,\chi'_i)$ converges to $g$, say. Then we have 
\[\chi'=\lim_i\chi'_i=\lim_i(\chi_i\cdot\tau(\chi_i,\chi'_i))=\chi\cdot g,
\] which implies $\tau(\chi,\chi')=g$ and $\tau(\chi_i,\chi'_i)\to\tau(\chi,\chi)$.
\end{proof}

\begin{remark}{\bf(Topological principal bundles).}\index{Bundles!Topological Principal}
The map $\tau$ in Proposition \ref{cont. action III} (c) is called the \emph{translation map} and is part of the definition of principal bundles in [Hu75]. Note that, if a topological group $G$ acts freely, continuously and satisfies (c), then $G$ automatically acts properly; thus the principal bundles in [Hu75] are by definition the free and proper $G$-spaces. Note further that these principal bundles are, in general not, locally trivial. We call a free and proper $G$-space which is Hausdorff a \emph{topological principal bundle}, if each orbit of the action is homeomorphic to $G$ and the orbit space is Hausdorff. For more informations on topological (locally trivial) principal bundles we refer to [RaWi98], Chapter 4, Section 2.
\end{remark}

\begin{theorem}\label{shit theorem}
Let $A$ be a commutative CIA, $G$ a compact group and $(A,G,\alpha)$ a free dynamical system. Then the induced action $\sigma:\Gamma_A\times G\rightarrow\Gamma_A$ is continuous and we obtain a topological principal bundle
\[(\Gamma_A,\Gamma_A/G,G,\sigma,\pr).
\]
\end{theorem}

\begin{proof}
\,\,\,We first recall from Remark \ref{equicont} that $\Gamma_A$ is equicontinuous. Hence, Lemma \ref{cont. action I} and Proposition \ref{cont. action II} imply that the map $\sigma$ is continuous. Further, we note that the map $\sigma$ is proper. Indeed, this follows from the compactness of $G$. In view of Theorem \ref{freeness of induced action}, we conclude that $\sigma$ is free. Therefore, $\Gamma_A$ is a free and proper $G$-space which is Hausdorff and thus the claim follows from Proposition \ref{cont. action III} (a) and (b). 
\end{proof}

\begin{corollary}
Let $A$ be a commutative CIA and $G$ a compact abelian group. Furthermore, let $(A,G,\alpha)$ be a dynamical system. If each isotypic component $A_{\varphi}$ contains an invertible element, then we obtain a topological principal bundle $(\Gamma_A,\Gamma_A/G,G,\sigma,\pr)$.
\end{corollary}

\begin{proof}
\,\,\,This assertion immediately follows from Proposition \ref{freeness for compact abelian groups II} and Theorem \ref{shit theorem}.
\end{proof}

\section{An Open Problem and an Application to the Structure Theory of $C^*$-Algebras}

This short section is dedicated to the following interesting open problem and the resulting application to the generalized Effros--Hahn conjecture:

\begin{open problem}{\bf (Primitive ideals).}\index{Primitive Ideals}
Theorem \ref{freeness of induced action} may be viewed as a first step towards a geometric approach to noncommutative principal bundles. Nevertheless, in order to get a broader picture, it might be helpful to get rid of the characters. This might be done with the help of primitive ideals, i.e., kernels of irreducible representations $(\rho,W)$ of the (locally convex) algebra $A$, since they can be considered as generalizations of characters (points). To be more precise:

Let $(A,G,\alpha)$ be a dynamical system, consisting of a (not necessarily commutative) unital locally convex algebra $A$, a topological group $G$ and a group homomorphism $\alpha:G\rightarrow\Aut(A)$, which induces a continuous action of $G$ on $A$. Further, let $\Prim(A)$\sindex[n]{$\Prim(A)$} denote the set of primitive ideals of $A$. As already mentioned, note that if $A$ is commutative, then $\Prim(A)\cong\Gamma_A$. Do there exist ``geometrically oriented" conditions which ensure that the corresponding action
\[\sigma:\Prim(A)\times G\rightarrow\Prim(A),\,\,\,(I,g)\mapsto\alpha(g).I
\]of $G$ on the primitive ideals $\Prim(A)$ of $A$ is free? For this purpose, it seems to be useful to study the paper [Ph09]
\end{open problem}

An interesting application to the structure theory of $C^*$-algebras is given by the ``generalized Effros--Hahn Conjecture":

\begin{theorem}{\bf(The generalized Effros--Hahn conjecture).}\index{Theorem!of Effros--Hahn}
Suppose $G$ is an amenable group, $A$ a separable $C^*$-algebra and $(A,G,\alpha)$ a $C^*$-dynamical system. If $G$ acts freely on $\Prim(A)$,  then there is one and only one primitive ideal of the crossed product $A\rtimes_{\alpha}G$ lying over each hull-kernel quasi-orbit in $\Prim(A)$. In particular, if every orbit is also hull-kernel dense, then $A\rtimes_{\alpha}G$ is simple.
\end{theorem}
 
\begin{proof}
\,\,\,A nice proof of this theorem can be found in [GoRo79], Corollary 3.3.
\end{proof}

\section{An Application to Noncommutative Geometry: Towards a Geometric Approach to Noncommutative Principal Bundles}\label{outlook}

The Theorem of Serre and Swan (cf. [Swa62]) justifies to consider finitely generated projective modules over unital algebras as ``noncommutative vector bundles ". Unfortunately, the case of principal bundles is not treated in the same satisfactory way. From a geometrical point of view it is, so far, not sufficiently well understood what a ``noncommutative principal bundle" should be. Still, there are several approaches towards the noncommutative geometry of principal bundles: For example, there is a well-developed abstract algebraic approach known as Hopf-Galois extensions which uses the theory of Hopf algebras (cf. [Sch04]). Another topologically oriented approach can be found in [ENOO09]; here the authors use $C^*$-algebraic methods to develop a theory of principal noncommutative torus bundles. In [Wa11a] we have developed a geometrically oriented approach to the noncommutative geometry of principal bundles based on dynamical systems and the representation theory of the corresponding transformation groups. We give a brief outlook:

According to Example \ref{induced transformation triples} each principal bundle $(P,M,G,q,\sigma)$ gives rise to a smooth dynamical system $(C^{\infty}(P),G,\alpha)$, consisting of the Fr\'echet algebra of smooth functions on the total space $P$, the structure group $G$ and a group homomorphism $\alpha:G\rightarrow\Aut(C^{\infty}(P))$, induced by the smooth action map $\sigma:P\times G\rightarrow P$ of $G$ on $P$. Conversely, given a manifold $P$ and a Lie group $G$, we have seen in Proposition \ref{smoothness of the group action on the set of characters} that each smooth dynamical system $(C^{\infty}(P),G,\alpha)$ gives rise to a smooth group action $\sigma:P\times G\rightarrow P$ and it is reasonable to ask if there exist natural (algebraic) conditions on $(C^{\infty}(P),G,\alpha)$ which ensure the freeness of $\sigma$. That is where Section \ref{section:free dynamical systems} enters the picture: If the smooth dynamical system $(C^{\infty}(P),G,\alpha)$ is free and if the induced action is additionally proper (which is for example the case if $G$ is compact), then Theorem \ref{free action theorem} implies that we obtain a locally trivial principal bundle of the form $(P,P/G,G,\pr,\sigma)$, where $\pr:P\rightarrow P/G,\,\,\,p\mapsto p.G$ denotes the corresponding orbit map. 

Now, it is interesting to investigate on the geometric structure of the induced principal bundle $(P,P/G,G,\pr,\sigma)$. For example, if we apply the results of Section \ref{free dyn sys casg} to the case $G=\mathbb{T}^n$, it turns out that the induced principal torus bundle $(P,P/\mathbb{T}^n,\mathbb{T}^n,\pr,\sigma)$ is trivial if and only if the corresponding isotypic components contain invertible elements. This observation justifies to call a dynamical system $(A,\mathbb{T}^n,\alpha)$ a \emph{trivial noncommutative principal $\mathbb{T}^n$-bundle} if the isotypic components contain invertible elements. While in classical (commutative) differential geometry there exists up to isomorphy only one trivial principal $\mathbb{T}^n$-bundle over a given manifold $M$, the situation completely changes in the noncommutative world. In particular, we provide a complete classification of all possible trivial noncommutative principal $\mathbb{T}^n$-bundles up to completion in terms of a suitable cohomology theory (cf. [Wa11b]). 

The step from the trivial to the non-trivial case is then carried out by introducing an appropriated method of localizing algebras in a ``smooth" way (cf. [Wa11c]) and saying that a (smooth) dynamical system $(A,\mathbb{T}^n,\alpha)$ is a \emph{NCP $\mathbb{T}^n$-bundle} if ``localization" around characters of the fixalgebra $Z$ of the induced action of $\mathbb{T}^n$ on the center $C_A$ of $A$ leads to trivial NCP $\mathbb{T}^n$-bundles. In particular, this approach covers the classical theory of principal $\mathbb{T}^n$-bundles and further examples are given by sections of algebra bundles endowed with certain actions of $\mathbb{T}^n$ by algebra automorphisms. 



\begin{appendix}

\section{Rudiments on the Smooth Exponential Law}\label{Rudiments on the Smooth Exponential Law}

For arbitrary sets $X$ and $Y$, let $Y^X$ be the set of all mappings from $X$ to $Y$. Then the following ``exponential law" holds: For any sets $X,Y$ and $Z$, we have
\[Z^{X\times Y}\cong(Z^Y)^X
\]as sets. To be more specific, the map
\[Z^{X\times Y}\cong(Z^Y)^X,\,\,\,f\rightarrow f^{\vee}
\]is a bijection, where
\[f^{\vee}:X\rightarrow Z^Y,\,\,\,f^{\vee}(x):=f(x,\cdot).
\]The inverse map is given by
\[(Z^Y)^X\rightarrow Z^{X\times Y},\,\,\,g\mapsto g^{\wedge},
\]where
\[g^{\wedge}:X\times Y\rightarrow Z,\,\,\,g^{\wedge}(x,y):=g(x)(y).
\]Next we consider an important topology for spaces of smooth functions:

\begin{definition}\label{smooth compact open topology}{\bf(Smooth compact open topology).}\index{Smooth!Compact Open Topology}
If $M$ is a locally convex manifold and $E$ a locally convex space, then the \emph{smooth compact open topology} on $C^{\infty}(M,E)$ is defined by the embedding
\[C^{\infty}(M,E)\hookrightarrow\prod_{n\in\mathbb{N}_0}C(T^nM,T^nE),\,\,\,f\mapsto(T^nf)_{n\in\mathbb{N}_0},
\]where the spaces $C(T^nM,T^nE)$ carry the compact open topology. Since $T^nE$ is a locally convex space isomorphic to $E^{2^n}$, the spaces $C(T^nM,T^nE)$ are locally convex and we thus obtain a locally convex topology on $C^{\infty}(M,E)$.
\end{definition}

It is natural to ask if there exists an ``exponential law"  for the space of smooth functions endowed with the smooth compact open topology, i.e., if we always have
\begin{align}
C^{\infty}(M\times N,E)\cong C^{\infty}(M,C^{\infty}(N,E))\label{exp law}
\end{align}
as a locally convex space, for all locally convex smooth manifolds $M$, $N$ and locally convex spaces $E$. In general, the answer is \emph{no}. Nevertheless, it can be shown that (\ref{exp law}) holds if $N$ is a finite-dimensional manifold over $\mathbb{R}$. To be more precise:

\begin{lemma}\label{smooth exp law}
Let $M$ be a real locally convex smooth manifold, $N$ a finite-dimensional smooth manifold and $E$ a topological 
vector space. Then a map $f:M\rightarrow C^{\infty}(N,E)$ is smooth if and only if the map
\[f^{\wedge}:M\times N\rightarrow E,\,\,\,f^{\wedge}(m,n):=f(m)(n)
\]is smooth. Moreover, the following map is an isomorphism of locally convex spaces:
\[C^{\infty}(M,C^{\infty}(N,E))\rightarrow C^{\infty}(M\times N,E),\,\,\,f\mapsto f^{\wedge}.
\]
\end{lemma}

\begin{proof}
\,\,\,A proof can be found in [NeWa07], Appendix, Lemma A3.
\end{proof}

\section{The Spectrum of the Algebra of Smooth Function}

In this part of the appendix we discuss the spectrum of the algebra of smooth functions. The proof of the following proposition originates from a unpublished paper of H. Grundling and K.-H. Neeb:

\begin{theorem}\label{spec of C(M,R) set}
Let $M$ be a manifold and let $C^{\infty}(M,\mathbb{R})$ be the unital Fr\'echet algebra of smooth functions on $M$. Then the following assertions hold:
\begin{itemize}
\item[\emph{(a)}]
Each closed maximal ideal of $C^{\infty}(M,\mathbb{R})$ is the kernel of an evaluation homomorphism \[\delta_m:C^{\infty}(M,\mathbb{R})\rightarrow\mathbb{R},\,\,\,f\mapsto f(m)\,\,\,\text{for some}\,\,\,m\in M.
\]
\item[\emph{(b)}]
Each character $\chi:C^{\infty}(M,\mathbb{R})\rightarrow\mathbb{R}$ is an evaluation in some point $m\in M$.
\end{itemize}
\end{theorem}
\begin{proof}
\,\,\,(a) Let $I\subseteq C^{\infty}(M,\mathbb{R})$ be a closed maximal ideal. If all functions vanish in the point $m$ in $M$, then the maximality of $I$ implies that $I=\ker\delta_m$. So we have to show that such a point exists. Let us assume that this is not the case. From that we shall derive the contradiction $I=C^{\infty}(M,\mathbb{R})$:

(i) Let $K\subseteq M$ be a compact set. Then for each $m\in K$ there exists a function $f_m\in I$ with $f_m(m)\neq 0$. The family $(f_m^{-1}(\mathbb{R}^{\times}))_{m\in K}$ is an open cover of $K$, so that there exist $m_1,\ldots,m_n\in K$ and a smooth function $f_K:=\sum_{i=1}^n f_{m_i}^2>0$ on $K$.

(ii) If $M$ is compact, then we thus obtain a function $f_M\in I$ which is nowhere zero. This leads to the contradiction $f_M\in C^{\infty}(M,\mathbb{R})^{\times}\cap I$. Next, suppose that $M$ is non-compact. Then there exists a sequence $(M_n)_{n\in N}$ of compact subsets with $M=\bigcup_n M_n$ and $M_n\subseteq M_{n+1}^0$. Let $f_n\in I$ be a non-negative function supported by $M_{n+1}\backslash M_{n-1}^0$ with $f_n>0$ on the compact set $M_n\backslash M_{n-1}^0$. Here the requirement on the support can be achieved by multiplying with a smooth function supported by $M_{n+1}\backslash M_{n-1}^0$ which equals 1 on $M_n\backslash M_{n-1}^0$. Then $f:=\sum_n f_n$ is a smooth function in $\overline{I}=I$ with $f>0$. Hence, $f$ is invertible, which is a contradiction.

(b) The proof of this assertion is divided into four parts:

(i) Let $\chi:C^{\infty}(M,\mathbb{R})\rightarrow\mathbb{R}$ be a character. If $f\in C^{\infty}(M,\mathbb{R})$ is non-negative, then for each $c>0$ we have $f+c=h^2$ for some smooth function $h\in C^{\infty}(M,\mathbb{R})^{\times}$, and this implies that $\chi(f)+c=\chi(f+c)\geq 0$, which leads to $\chi(f)\geq -c$, and consequently to $\chi(f)\geq 0$.

(ii) Now, we choose a smooth function $F:M\rightarrow\mathbb{R}$ for which the sets $F^{-1}(]-\infty,c])$, $c\in\mathbb{R}$, are compact. Such a function can easily be constructed from a sequence $(M_n)_{n\in\mathbb{N}}$ as above.

(iii) We consider the ideal $I=\ker\chi$. If $I$ has a zero, then $I=\ker\delta_m$ for some $m$ in $M$ and this implies $\chi=\delta_m$. Hence, we may assume that $I$ has no zeros. Then the argument of (a) provides for each compact subset $K\subseteq M$ a compactly supported function $f_K\in I$ with $f_K>0$ on $K$. If $h\in C^{\infty}(M,\mathbb{R})$ is supported by $K$, we therefore find a $\lambda>0$ with $\lambda f_K-h\geq 0$, which leads to
\[0\leq\chi(\lambda f_K-h)=\chi(-h),
\]and hence to $\chi(h)\leq0$. Replacing $h$ by $-h$, we also get $\chi(h)\geq 0$ and hence $\chi(h)=0$. Therefore, $\chi$ vanishes on all compactly supported functions.

(iv) For $c>0$ we now pick a non-negative function $f_c\in I$ with $f_c>0$ on the compact subset $F^{-1}(]-\infty,c])$. Then there exists a $\mu>0$ with $\mu f_c+F\geq c$ on $F^{-1}(]-\infty,c])$. Now $\mu f_c+F\geq c$ holds on all of  $M$, and therefore $\chi(F)=\chi(\mu f_c+F)\geq c$.
Since $c>0$ was arbitrary, we arrive at a contradiction.
\end{proof}

\begin{corollary}\label{spec of C(M,K) set}
Let $M$ be a manifold and let $C^{\infty}(M,\mathbb{C})$ be the unital Fr\'echet algebra of smooth complex-valued functions on $M$. Then each character $\chi:C^{\infty}(M,\mathbb{C})\rightarrow\mathbb{C}$ is an evaluation in some point $m\in M$.
\end{corollary}

\begin{proof}
\,\,\,This assertion easily follows from Theorem \ref{spec of C(M,R) set} (b): Indeed, we just have to note that each element $f\in C^{\infty}(M,\mathbb{C})$ can be written as a sum of smooth real-valued functions $f_1$ and $f_2$, i.e., 
\[f=f_1+if_2\,\,\,\text{for}\,\,\,f_1,f_2\in C^{\infty}(M,\mathbb{R}),
\]and that the restriction of the character $\chi$ to $C^{\infty}(M,\mathbb{R})$ is real-valued, i.e., defines a character of $C^{\infty}(M,\mathbb{R})$.
\end{proof}

The preceding proposition shows that the correspondence between $M$ and $\Gamma_{C^{\infty}(M)}$ is actually a topological isomorphism:

\begin{proposition}\label{spec of C(M) top}
Let $M$ be a manifold. Then the map
\begin{align}
\Phi_M:M\rightarrow \Gamma_{C^{\infty}(M)},\,\,\,m\mapsto \delta_m\notag.
\end{align}
is a homeomorphism.
\end{proposition}

\begin{proof}
\,\,\,(i) The surjectivity of $\Phi$ follows from Corollary \ref{spec of C(M,K) set}.
To show that $\Phi$ is injective, choose elements $m\neq m'$ of $M$. Since $M$ is manifold, there exists a function $f$ in $C^{\infty}(M)$ with $f(m)\neq f(m')$. Then 
\[\delta_m(f)=f(m)\neq f(m')=\delta_{m'}(f)
\]implies that $\delta_m\neq\delta_{m'}$, i.e., $\Phi$ is injective. 

(ii) Next, we show that $\Phi$ is continuous: Let $m_n\rightarrow m$ be a convergent sequence in $M$. Then we have 
\[\delta_{m_n}(f)=f(m_n)\rightarrow f(m)=\delta_m(f)\,\,\,\text{for all}\,\,\,f\,\,\,\text{in}\,\,\,C^{\infty}(M),
\]i.e., $\delta_{m_n}\rightarrow\delta_m$ in the topology of pointwise convergence. Hence, $\Phi$ is continuous. 

(iii) We complete the proof by showing that $\Phi$ is an open map: For this let $U$ be an open subset of $M$, $m_0$ in $U$ and $h$ a smooth real-valued function with $h(m_0)\neq 0$ and $\supp(h)\subset U$. Since the map
\[\delta_h:\Gamma_{C^{\infty}(M)}\rightarrow\mathbb{K},\,\,\,\delta_m\mapsto h(m)
\]is continuous, a short calculations shows that $\Phi(U)$ is a neighbourhood of $m_0$ containing the open subset $\delta_h^{-1}(\mathbb{K}^{\times})$. Hence, $\Phi$ is open.
\end{proof}

\end{appendix}

\printindex[n] 
\printindex[idx] 
\end{document}